\documentclass{amsart}
\usepackage{geometry} 
\usepackage{url}
\usepackage{amsthm}
\usepackage{amssymb}
\usepackage{latexsym}
\usepackage{amsfonts}
\usepackage{amsmath}
\usepackage{amstext}
\usepackage{accents}
\usepackage{cancel}
\usepackage{color}
\usepackage{enumerate}
\usepackage{eucal}
\usepackage{amscd}
\usepackage{hyperref}
\usepackage{graphicx}
\usepackage{verbatim}

\newtheorem{theorem}{Theorem}
\newtheorem{lemma}[theorem]{Lemma}
\newtheorem{remark}[theorem]{Remark}
\newtheorem{corollary}[theorem]{Corollary}
\newtheorem{example}[theorem]{Example}

\newtheorem{proposition}[theorem]{Proposition}

\newtheorem{definition}[theorem]{Definition}

\newcommand{\R}{\mathbb{R}}
\renewcommand{\S}{\mathbb{S}}

\newcommand{\IP}[2]{\ensuremath{\left\langle #1, #2 \right\rangle}}

\begin{document}

\title[Solitons for the inverse mean curvature flow]{Solitons for the inverse mean curvature 
flow}

\author[G. Drugan]{Gregory Drugan}
\address{University of Oregon}
\email{drugan@uoregon.edu}
\author[H. Lee]{Hojoo Lee}
\address{Korea Institute for Advanced Study}
\email{autumn@kias.re.kr, momentmaplee@gmail.com}
\author[G. Wheeler]{Glen Wheeler}
\address{Institute for Mathematics and its Applications, University of Wollongong}
\email{glenw@uow.edu.au}

\begin{abstract}
We investigate self-similar solutions to the inverse mean curvature flow in
Euclidean space. In the case of one dimensional planar solitons, we explicitly classify all homothetic solitons and translators. 
 Generalizing Andrews' theorem that circles are the only compact homothetic planar solitons, we apply the Hsiung--Minkowski integral formula  to
prove the rigidity of the hypersphere in the class of compact expanders of codimension
one. We also establish that the moduli space of compact expanding surfaces of
codimension two is big. Finally, we update the list of Huisken--Ilmanen's
rotational expanders by constructing new examples of complete expanders with
rotational symmetry, including topological hypercylinders, called \textit{infinite bottles},
that interpolate between two concentric round hypercylinders.
\end{abstract}

\maketitle


\section{Main results}

In this paper, we study self-similar solutions to the inverse mean curvature flow in
Euclidean space. After a brief introduction, we present an explicit classification of the one dimensional homothetic solitons (Theorem~\ref{ALLcurves}). Examples include circles, involutes of circles, and logarithmic spirals. Then, we prove that families of cycloids are the only translating solitons (Theorem \ref{TMcycloid}), and we show how to construct translating surfaces via a tilted product of cyloids. 

Next, we consider the rigidity of homothetic solitons. In the class of closed homothetic solitons of codimension one, we prove that the round hyperspheres are rigid (Theorem~\ref{rigid1}).  For the higher codimension case, we observe that any minimal submanifold of the standard hypersphere is an expander, so in light of Lawson's construction \cite{Lawson1970} of minimal surfaces in $\S^3$, there are compact embedded expanders for any genus in ${\mathbb{R}}^{4}$. 

We conclude with an investigation of homothetic solitons with rotational symmetry. First, we construct new examples of complete expanders with rotational symmetry, called \emph{infinite bottles}, which are topological hypercylinders that interpolate between two concentric round hypercylinders (Theorem~\ref{chimney_thm}). Then, we show how the analysis in the proof of Theorem~\ref{chimney_thm} can be used to construct other examples of complete expanders with rotational symmetry, including the examples from Huisken-Ilmanen~\cite{HI1997}.

\section{Inverse mean curvature flow - history and applications} \label{intro}

Round hyperspheres in Euclidean space expand under the inverse mean curvature flow (IMCF) with an exponentially increasing radius. This behavior is typical for the flow. Gerhardt \cite{G1990} and Urbas \cite{U1990} showed that compact, star-shaped initial hypersurfaces with strictly positive mean curvature converge under IMCF, after suitable rescaling, to a round sphere.

Strictly positive mean curvature is an essential condition. For the IMCF to be parabolic, the mean curvature must be strictly positive.
Huisken and Ilmanen \cite{HI2008} proved that smoothness at later times is
characterised by the mean curvature remaining bounded strictly away from zero
(see also Smoczyk \cite{S2000}).
Within the class of strictly mean-convex surfaces, however, a solution to
inverse mean curvature flow will, in general, become singular in finite time. For
example, starting from a thin embedded torus with positive mean curvature in ${\mathbb{R}}^3$, the surface fattens up under IMCF and, after finite time, the mean curvature reaches zero at some points \cite[p. 364]{HI2001}. Thus, the classical description breaks down, and any appropriate weak definition of inverse mean curvature flow would need
to allow for a change of topology.

In 2001 Huisken and Ilmanen \cite{HI2001} used a level-set approach and
developed the notion of weak solutions for IMCF to overcome theses problems.
They showed existence for weak solutions and proved that Geroch''s monotonicity \cite{G1973}
for the Hawking mass carries over to the weak setting.
This enabled them to prove the Riemannian Penrose inequality, which also gave an
alternative proof for the Riemannian positive mass theorem. For a summary, we refer the reader to Huisken-Ilmanen \cite{HI1997,HI1997_2}. The work of Huisken and Ilmanen also shows that weak solutions become star-shaped and smooth outside
some compact region and thus (by the results of Gerhardt \cite{G1990} and Urbas \cite{U1990}) round in the limit.
Using a different geometric evolution equation, Bray \cite{B2001} proved the
most general form of the Riemannian Penrose inequality.
An overview of the different methods used by Huisken, Ilmanen, and Bray
can be found in \cite{B2002}.
An approach to solving the full Penrose inequality involving a generalised
inverse mean curvature flow was proposed in \cite{BHMS2007}.
To our knowledge, the full Penrose inequality is still an open problem.

Finally, let us mention some other applications and new developments in IMCF. Using IMCF, 
 Bray and Neves \cite{BN2004} proved the Poincar\'e
conjecture for 3-manifolds with Yamabe invariant greater than that of
$\R\mathbb{P}^3$ (see also \cite{AN2007}). 
Connections with $p$-harmonic functions and the weak formulation
of inverse mean curvature flow are described in \cite{M2007}, where a new proof for the existence of a proper weak solution is given, and in \cite{LWW2011}, where gradient bounds and non-existence results are proved. 
Recently, Kwong and Miao \cite{KM2014}
discovered a monotone quantity for the IMCF, which they used to derive new geometric inequalities for star-shaped hypersurfaces with positive mean curvature.


\section{Definitions and one dimensional examples} 
\label{SSsec2}

\begin{definition}[\textbf{Homothetic solitons of arbitrary codimension}] \label{SOL} 
A submanifold  ${\Sigma}^{n} \subset {\mathbb{R}}^{N}$ with nonvanishing mean curvature vector field $\overrightarrow{\, H \,}$ is called a \textit{homothetic soliton for the inverse mean curvature flow} if there exists a constant $C \in \mathbb{R}-\{0\}$ satisfying  
\begin{equation} \label{SOLequation}
  -  \frac{1}{{\vert \overrightarrow{\, H \,} \vert}^{2} \; } \overrightarrow{\, H \,}  = C  {X}^{\perp} \quad \; \text{on}  \; {\Sigma}, 
\end{equation}
where the vector field ${X}^{\perp}$ denotes the normal component of $X$. We notice that, for any constant $\lambda \neq 0$, the rescaled immersion $\lambda X$ is a soliton with the same value of $C$.  
\end{definition}

\begin{remark}  \label{support function}
On a homothetic soliton ${\Sigma}^{n} \subset {\mathbb{R}}^{N}$, we observe that the condition (\ref{SOLequation}) implies  
\[
 {\vert \overrightarrow{\, H \,} \vert}^{2}  = \IP { \overrightarrow{\, H \,}   }{\overrightarrow{\, H \,}  } 
=\IP { -  C {\vert \overrightarrow{\, H \,} \vert}^{2}     {X}^{\perp}  }{\overrightarrow{\, H \,}  }  =
- C   {\vert \overrightarrow{\, H \,} \vert}^{2}    \IP {  {X} } {  \overrightarrow{\, H \,}  }. 
\]
Since the mean curvature vector field $\overrightarrow{\, H \,}$ is nonvanishing, this shows

\[
 - \IP{\overrightarrow{\, H \,} }{X} =   \frac{1}{C} \quad \text{or} \quad  - \IP{ {\triangle}_{g} X   }{X} =   \frac{1}{C} 
  \quad \text{or} \quad {\triangle}_{g}  {\vert X \vert}^{2}    =  2 \left(   n - \frac{1}{C}  \right). 
\]
where $g$ denotes the induced metric on $\Sigma$.
\end{remark}

\begin{proposition}[\textbf{Homothetic solitons of codimension one}]  \label{SOL2}
Let  ${\Sigma}^{n} \subset {\mathbb{R}}^{n+1}$ be an oriented hypersurface with nowhere vanishing mean curvature vector field $\overrightarrow{\, H \,}= {\triangle}_{g} X$. Then, it becomes a \textit{homothetic soliton to the inverse mean curvature flow} if and only if there exists a constant $C \in \mathbb{R}-\{0\}$ satisfying  
\begin{equation} \label{SOLhyper2}
 - \IP{\overrightarrow{\, H \,} }{X} =   \frac{1}{C} \quad \text{or equivalently} \quad  - \IP{ {\triangle}_{g} X   }{X} =   \frac{1}{C}.
\end{equation} 
\end{proposition}

\begin{proof} 
According to the observation in Remark \ref{support function}, the vector equality in (\ref{SOLequation}) implies the scalar equality in (\ref{SOLhyper2}). To see that (\ref{SOLhyper2}) implies (\ref{SOLequation}), let $\mathbf{N}$ denote a unit normal vector, and let $H=- \left(div_{{}_{\Sigma}} \mathbf{N} \right)$ be the corresponding scalar mean curvature. Then
$\overrightarrow{\, H \,}= {\triangle}_{g} X =H\mathbf{N}$, and the condition (\ref{SOLhyper2}) becomes
\[
   - \IP{ H\mathbf{N} }{X} =   \frac{1}{C},
\]
which implies
\[   
C  {X}^{\perp} = \IP{\mathbf{N} }{CX} \mathbf{N} = -  \frac{1}{H} \mathbf{N}  =-  \frac{1}{{H}^2 }  \overrightarrow{\, H \,}.
\]
\end{proof}

\subsection{Expanders and shrinkers}    \label{Classification of homothetic soliton curves}

In 2003, Andrews \cite[Theorem 1.7]{Andrews2003} proved that circles centered
at the origin are the only \textit{compact} homothetic solitons for the inverse
mean curvature flow in ${\mathbb{R}}^{2}$. We give an explicit classification
of all homothetic soliton curves. In particular, the classical logarithmic
spirals and involutes of circles become expanders.

\begin{theorem}[\textbf{Curvature on homothetic soliton curves}] \label{geometric characterizations}
Let  $\mathcal{C}$ be a homothetic soliton curve with the velocity constant  $c \in \mathbb{R}-\{0\}$ for the inverse curve shortening flow.  Then, its curvature function $\kappa$ satisfies the Poisson equation 
\begin{equation}  \label{curvature arc length 1}
  {\triangle}_{\mathcal{C}}  \; \frac{1}{ \; {\kappa}^{2} \; } = 2 ( c - 1).
\end{equation}
This guarantees the existence of constants ${\alpha}_{1}$, ${\alpha}_{2} \in \mathbb{R}$ satisfying
\begin{equation}  \label{curvature arc length 2}
   {{\kappa}^{2}} =  \frac{1}{ (c-1) s^2 +  {\alpha}_{1} s + {\alpha}_{2} },
\end{equation}
where $s$ denotes an arc length parameter on  the soliton curve $\mathcal{C}$.
\end{theorem}

\begin{proof} We begin with a unit speed patch $X(s)=\left(x(s), y(s)\right)$ of the curve $\mathcal{C}$. The unit tangent vector $T(s)$, the unit normal vector $N(s)$, and the tangential angle map $\theta(s)$ are defined by
\[
  T(s)= \left(\, \dot{x}(s), \dot{y}(s) \, \right) = \left( \, \cos \theta(s), \sin \theta(s) \, \right), 
\quad   N(s)= \left(\, - \dot{y}(s), \dot{x}(s) \, \right).
\]
The curvature vector $\overrightarrow{\kappa}$ and scalar curvature $\kappa$ are given by
\[
\overrightarrow{\kappa}(s) =  \left(\, \ddot{x}(s), \ddot{y}(s) \, \right) = \kappa(s)  N(s), \quad  \kappa(s)  = \dot{\theta} (s).
\] 
Introducing $\tau=X \cdot T$ and $\nu =X \cdot N$, we have the well-known structure equations
\[
  \frac{d\tau}{ds} = 1+ \kappa \nu, \quad \frac{d\nu}{ds} = - \kappa \tau.
\]
Since $\mathcal{C}$ is a homothetic soliton curve with speed $c \in \mathbb{R}-\{0\}$, we have 
$ -  \overrightarrow{\kappa} \cdot X  = \frac{1}{c} \; \Longleftrightarrow \; \kappa \nu = - \frac{1}{c}$.
In particular, $\nu$ and $\frac{d\theta}{ds}=\kappa$ are nonvanishing. We can rewrite the structure equations as 
\[
\frac{d\tau}{d\theta} = (1-c) \nu, \quad \frac{d\nu}{d\theta} = -  \tau,
\]
which implies the ODE
\begin{equation}  \label{support 1D}
  \frac{d^{2}\nu}{{d\theta}^{2}}  + (1-c) \nu = 0.
\end{equation}
Noticing that 
$\frac{d^{2} }{{ds}^{2}}    =  \kappa \frac{d}{d \theta} \,  \left( \;    \kappa \frac{d}{d \theta}  \; \right)
= \frac{1}{c^2v} \frac{d}{d \theta} \,  \left( \;    \frac{1}{v}   \frac{d}{d \theta}  \; \right)$ and using (\ref{support 1D}), 
we conclude 
\[
  \frac{d^{2}}{{ds}^{2}} \left( \frac{1}{{\kappa}^{2}} \right) 
  = \frac{1}{c^2v} \frac{d}{d \theta} \,  \left( \;    \frac{1}{v}   \frac{d}{d \theta} \, \left( c^2 v^2 \right) \; \right)
= \frac{2}{v}   \frac{d^{2}\nu}{{d\theta}^{2}} = 2(c-1). 
\]
\end{proof}

\begin{theorem}[\textbf{Explicit parametrization of homothetic soliton curves}] \label{ALLcurves}
Let  $\mathcal{C}$ be a homothetic soliton curve with constant  $c \in \mathbb{R}-\{0\}$ for the inverse curve shortening flow. Then there exist  constants ${\mu}_{1}, {\mu}_{2} \in \mathbb{R}$ such that  the curve $\mathcal{C}$ admits an explicit patch $X_{\left({\mu}_{1}, {\mu}_{2}\right)}=\left( \, x(\theta), y(\theta) \, \right):$
\begin{enumerate}
\item $c<0$ or $0<c<1:$ Set $\alpha=\sqrt{1-c}$ 
\[
\begin{cases}
 {x}_{\left({\mu}_{1}, {\mu}_{2}\right)}(\theta) =    \alpha \left[ {\mu}_{1}  \sin \left( \alpha   \theta  \right)  - {\mu}_{2}  \cos \left( \alpha   \theta  \right)    \right] \cos \theta  - \left[ {\mu}_{1}  \cos \left( \alpha   \theta  \right)   + {\mu}_{2}  \sin \left( \alpha   \theta  \right)    \right] \sin \theta     , \\ 
 {y}_{\left({\mu}_{1}, {\mu}_{2}\right)}(\theta) =   \alpha \left[ {\mu}_{1}  \sin \left( \alpha   \theta  \right)   - {\mu}_{2}  \cos \left( \alpha   \theta  \right)    \right] \sin \theta + \left[ {\mu}_{1}  \cos \left( \alpha   \theta  \right)   + {\mu}_{2}  \sin \left( \alpha   \theta  \right)    \right] \cos \theta   .
\end{cases}
\]
\item $c=1:$  
\[
\begin{cases}
 {x}_{\left({\mu}_{1}, {\mu}_{2}\right)}(\theta) =   -{\mu}_{2} \cos \theta - ( {\mu}_{1} + {\mu}_{2} \theta )  \sin \theta,   , \\ 
 {y}_{\left({\mu}_{1}, {\mu}_{2}\right)}(\theta) =  -{\mu}_{2} \sin \theta + ( {\mu}_{1} + {\mu}_{2} \theta )  \cos \theta.
\end{cases}
\]
\item $c>1:$ Set $\alpha=\sqrt{c-1}$.  
\[ 
\begin{cases}
 {x}_{\left({\mu}_{1}, {\mu}_{2}\right)}(\theta) = -  \alpha \left[ {\mu}_{1}  \sinh \left( \alpha   \theta  \right)   + {\mu}_{2}  \cosh \left( \alpha   \theta  \right)    \right] \cos \theta  - \left[ {\mu}_{1}  \cosh \left( \alpha   \theta  \right)   + {\mu}_{2}  \sinh \left( \alpha   \theta  \right)    \right] \sin \theta     , \\ 
 {y}_{\left({\mu}_{1}, {\mu}_{2}\right)}(\theta) =  - \alpha \left[ {\mu}_{1}  \sinh \left( \alpha   \theta  \right)   + {\mu}_{2}  \cosh \left( \alpha   \theta  \right)    \right] \sin \theta \, + \left[ {\mu}_{1}  \cosh \left( \alpha   \theta  \right)   + {\mu}_{2}  \sinh \left( \alpha   \theta  \right)    \right] \cos \theta   .
\end{cases}
\]
\end{enumerate}
\end{theorem}

\begin{proof} It is a continuation of the proof of Theorem \ref{geometric characterizations}. We first observe that 
\[
 X = \left(x, y \right) = \left(  \tau \cos \theta - \nu \sin \theta, \tau \sin \theta + \nu \cos \theta  \right)
=  \left(  - \frac{d\nu}{d\theta} \cos \theta - \nu \sin \theta, - \frac{d\nu}{d\theta} \sin \theta + \nu \cos \theta  \right).
\]
In addition, the ODE
\[
  \frac{d^{2}\nu}{{d\theta}^{2}}  + (1-c) \nu = 0
\]
can be explicitly integrable depending on the sign of $1-c$. We have the three cases:
\begin{enumerate}
\item $c<0$ or $0<c<1:$ Setting $\alpha=\sqrt{1-c}$, we have  
\[
 \nu =   {\mu}_{1}  \cos \left( \alpha   \theta  \right)   + {\mu}_{2}  \sin \left( \alpha   \theta  \right), \quad 
 \frac{d\nu}{d\theta} =  - \alpha \left[ {\mu}_{1}  \sin \left( \alpha   \theta  \right)  - {\mu}_{2}  \cos \left( \alpha   \theta  \right)    \right].
\]
\item $c=1:$  We have 
$\nu =   {\mu}_{1}     + {\mu}_{2}    \theta$ and  
 $\frac{d\nu}{d\theta} = {\mu}_{2}$.

\item $c>1:$ Setting $\alpha=\sqrt{c-1}$,  we have
\[
 \nu =   {\mu}_{1}  \cosh \left( \alpha   \theta  \right)   + {\mu}_{2}  \sinh \left( \alpha   \theta  \right), \quad 
 \frac{d\nu}{d\theta} = \alpha \left[ {\mu}_{1}  \sinh \left( \alpha   \theta  \right)   + {\mu}_{2}  \cosh \left( \alpha   \theta  \right)    \right].
\]
\end{enumerate}
\end{proof}

\begin{remark} We point out some classical cases among soliton curves.
\begin{enumerate}
\item The case $c=1$: When ${\mu}_{2}=0$, it is a circle of radius $\vert {\mu}_{1} \vert$. When ${\mu}_{2} \neq 0$. it becomes the involute of the circle of radius $\vert {\mu}_{2} \vert$.
\item The case $c>1$: Set $\alpha=\sqrt{c-1}=\tan \beta$ and take ${\mu}_{1} ={\mu}_{2}=1$, we have the soliton
\[
      \left( x(\theta), y(\theta) \right) = \frac{e^{ \left( \tan \beta \right) \theta} }{\cos \beta} \left(  \, \sin (\theta+\beta), \,
      - \cos (\theta+\beta) \, \right).
\]
Up to homotheties, reflections, and rotations, it is  the logarithmic spiral
$r = e^{ \left( \tan \beta \right) \theta}$. It is worth to mention the geometrical observation that 
logarithmic spirals could be regarded as generalized involutes of a single point. 
See \cite[Example 2]{AM2010}. 

\end{enumerate}
\end{remark}

\subsection{Cycloids as translators}

\begin{definition}[\textbf{Translators of arbitrary codimension}] \label{translatingSOL} 
A submanifold  ${\Sigma}^{n} \subset {\mathbb{R}}^{N}$ with nonvanishing mean curvature vector field $\overrightarrow{\, H \,}$ is called a \textit{translator for the inverse mean curvature flow} if there exists a non-zero constant vector field $\mathbf{V}$ satisfying  
\begin{equation} \label{translator equation}
  -  \frac{1}{{\vert \overrightarrow{\, H \,} \vert}^{2} \; } \overrightarrow{\, H \,}  =    {\mathbf{V}}^{\perp} \quad \; \text{on}  \; {\Sigma}, 
\end{equation}
where the vector field ${\mathbf{V}}^{\perp}$ denotes the normal component of $\mathbf{V}$. We say that 
 $\mathbf{V}$ is the velocity of the translator ${\Sigma}$.
\end{definition}

\begin{proposition}[\textbf{Translators of codimension one}]  \label{translatingSOL2}
Let  ${\Sigma}^{n} \subset {\mathbb{R}}^{n+1}$ be an oriented hypersurface with nonvanishing mean curvature vector field $\overrightarrow{\, H \,}= {\triangle}_{g} X$, where $g$ denotes the induced metric on $\Sigma$. Then, it becomes a \textit{translator to the inverse mean curvature flow} if and only if there exists a non-zero constant vector field $\mathbf{V}$ satisfying  
\begin{equation} \label{translatingSOLhyper2}
 \IP {     {\mathbf{V}}       }{\overrightarrow{\, H \,}  } = -1. 
 \end{equation} 
\end{proposition}

\begin{proof} We first observe that the 
condition (\ref{translator equation}) implies the equality
\[
 -1  = \IP {     -  \frac{1}{{\vert \overrightarrow{\, H \,} \vert}^{2} \;  }   \overrightarrow{\, H \,}       }{\overrightarrow{\, H \,}  } 
=\IP {     {\mathbf{V}}^{\perp}      }{\overrightarrow{\, H \,}  } 
=\IP {     {\mathbf{V}}       }{\overrightarrow{\, H \,}  }. 
\] 
It remains to check that the scalar equality
 (\ref{translatingSOLhyper2}) implies the vectorial equality in (\ref{translator equation}). Let $\mathbf{N}$ denote a unit normal vector and $H=- \left(div_{{}_{\Sigma}} \mathbf{N} \right)$ its  scalar mean curvature so that
$\overrightarrow{\, H \,}= {\triangle}_{g} X =H\mathbf{N}$.
Then the condition (\ref{translatingSOLhyper2}) becomes $-1 =  \IP {     \mathbf{V}     }{\overrightarrow{\, H \,}  } =  H  \IP {     {\mathbf{V}}       }{ \mathbf{N}  }$, 
which implies
\[   
 {\mathbf{V}}^{\perp} = \IP  {   \mathbf{V}  }{  \mathbf{N} } \mathbf{N} = -  \frac{1}{H} \mathbf{N}  =-  \frac{1}{{H}^2 }  \overrightarrow{\, H \,}.
\]
\end{proof}

\begin{corollary}[\textbf{Height function on translating hypersurfaces}]  \label{translatingSOL2b}
A submanifold ${\Sigma}^{n} \subset {\mathbb{R}}^{n+1}$ with nonvanishing mean curvature is a \textbf{translator to the inverse mean curvature flow} with velocity $\mathbf{V}
=(0, \cdots, 0, 1)$ if and only if 
\begin{equation} \label{translatingSOLhyper3}
-1 = {\triangle}_{{}_{\Sigma}} x_{n+1}  \quad \; \text{on}  \; {\Sigma}.
\end{equation}  
\end{corollary}

Now, we prove the uniqueness of cycloids as the one dimensional translator in  ${\mathbb{R}}^{2}$.

\begin{theorem}[\textbf{Classification of translating curves  in ${\mathbb{R}}^{2}$}]
\label{TMcycloid}
Any translating curves with unit speed for the inverse mean curvature flow in the Euclidean plane are congruent to cycloids generated by a circle of radius $\frac{1}{4}$. 
\end{theorem}

\begin{proof}
Let the connected curve $\mathcal{C}$ be a translator in the $xy$-plane with unit velocity $\mathbf{V}= (0, 1)$. 
Adopt the paparmetrization $X(s) = ( x(s), y(s) )$, where  
$s$ denotes the arclength on $\mathcal{C}$ and introduce the tangential angle function $\theta(s)$ such that
the tangent $\frac{dX}{ds}=\left( \cos \theta, \sin \theta \right) $ and
the normal $N(s)=\left( -\sin \theta, \cos \theta \right)$. The translator condition reads
\[
 - \frac{1}{\kappa} = \cos \theta
\]
Now, we integrate 
\[
 \left( \frac{d x}{d \theta},  \frac{d y}{d \theta} \right) =  \left( \frac{d s}{d \theta} \frac{dx}{ds}, \frac{d s}{d \theta}  \frac{d y}{d s} \right) =  \left( \frac{1}{\kappa} \cos \theta, \frac{1}{\kappa} \sin \theta \right)=  \left( -   {\cos}^{2} \theta, - \cos \theta  \sin \theta \right)
\]
to recover the curve, up to translations, 
\[
\left(  x, y  \right) =  \frac{1}{4} \left( - 2 \theta - \sin\left( 2 \theta \right)  ,  1+  \cos \left( 2 \theta \right)  \right).
\]
After introducing the new variable $t=-\pi+2\theta$, we have $\left(  x, y  \right) =  \frac{1}{4} \left( -\pi - t + \sin t,  1 -  \cos t \right)$.
Reflecting about the x-axis and the translating along the $(1,0)$ direction, the translator is congruent to the cycloid 
represented by $\frac{1}{4} \left(t - \sin t,  1 -  \cos t \right)$. Therefore, we conclude that  $\mathcal{C}$ is congruent to the cycloid through
 the origin, generated by a circle of radius $\frac{1}{4}$.
\end{proof}

\begin{example}[\textbf{Tilted cycloid products - one parameter family of translators with the same speed in ${\mathbb{R}}^{3}$}] We can use cycloids (one dimensional translator in ${\mathbb{R}}^{2}$) to construct a one parameter family of two dimensional translators with velocity $(0,0,1)$ in  ${\mathbb{R}}^{3}$. Let $(\alpha(s), \beta(s))$ denote a unit speed patch of the translating curve $\mathcal{C}$ with velocity $(0,1)$ in the $\alpha\beta$-plane, so that $\beta''(s)=-1$ on the translator $\mathcal{C}$. For each constant $\mu \in \left( -\frac{\pi}{2}, \frac{\pi}{2} \right)$, we introduce orthonormal vectors 
 \[
  {\textbf{v}}_{1}=(\cos \mu, \; 0, \; -\sin \mu ) , \quad {\textbf{v}}_{2}=(0, \; 1, \; 0 ) , \quad {\textbf{v}}_{3}=(\sin \mu , \; 0, \; \cos \mu ), 
  \]
   and associate the product surface 
 ${\Sigma}_{\mu} = \mathbb{R} \times \frac{1}{\cos  \mu} \mathcal{C}$ defined by the patch
 \[
 \mathbf{X}(s,h)= h {\textbf{v}}_{1} + \frac{ \alpha(s)}{\cos \mu} {\textbf{v}}_{2} + \frac{ \beta(s)}{\cos \mu}  {\textbf{v}}_{3}.
 \]
A straightforward computation yields 
 \[
  \IP{  \triangle_{{\Sigma}_{\mu}} X }{(0,0,1)} =\IP{ \frac{1}{\cos \mu}  \left(  \alpha''(s) {\textbf{v}}_{2}  + \beta''(s) {\textbf{v}}_{3}  \right)  }{(0,0,1)} =
  \beta''(s)=-1,
 \]
 which guarantees that the surface ${\Sigma}_{\mu}$ becomes a translator with velocity $(0,0,1)$ in ${\mathbb{R}}^{3}$.
 \end{example}


\section{Rigidity of hyperspheres and spherical expanders}
\label{SSsec3}
 
We first prove that hyperspheres, as homothetic solitons to the inverse mean curvature flow, are exceptionally rigid.
It is a higher dimensional generalization of Andrews' result  \cite[Theorem 1.7]{Andrews2003} that circles centered at the origin 
are the only compact homothetic solitons in  ${\mathbb{R}}^{2}$. 
We then explain that the moduli space of spherical expanders of higher codimension is large. 
 
\begin{theorem}[\textbf{Uniqueness of spheres as compact solitons}]  \label{rigid1}
Let ${\Sigma}^{n \geq 2}$ be a homothetic soliton hypersurface for the inverse mean curvature flow in  ${\mathbb{R}}^{n+1 \geq 3}$. If ${\Sigma}$ is closed, then it is a round hypersphere (centered at the origin). 
\end{theorem}

\begin{proof} Since $\Sigma$ is a compact hypersurface with nonvanishing mean curvature vector, there exists an inward pointing unit normal vector field ${\mathbf{N}}$ along $\Sigma$. Then $\overrightarrow{\, H \,}= {\triangle}_{g} X =H\mathbf{N}$, where the scalar mean curvature $H=-  div_{{}_{\Sigma}} \,  {\mathbf{N}}$ is positive. Since $\Sigma$ is a homothetic soliton, we have 
\begin{equation}   \label{rigidity proof 02}
  \frac{1}{C} = - \IP{X} {\overrightarrow{\, H \,} }=  - H \IP{X}{ \mathbf{N} }, 
\end{equation}
for some constant $C \neq 0$. The Hsiung--Minkowski formula \cite{Hsiung56} gives  
\[
   0 = \int_{\Sigma} \left(1 + \frac{1}{n} \IP{X}{\overrightarrow{\, H \,} } \right) d\Sigma
     = \left( 1 - \frac{1}{nC} \right)  \int_{\Sigma} 1 \; d\Sigma.
\]
It follows that $C=\frac{1}{n}$.  Let 
${\kappa}_{1}$, $\cdots$, ${\kappa}_{n}$ be principal curvature functions on $\Sigma$. 
In terms of
\[
{\sigma}_{2} = \frac{2}{n(n-1)} \sum_{1 \leq i < j \leq n} {\kappa}_{i} {\kappa}_{j} =  \frac{H^2}{n^2}  - \frac{1}{n^{2}(n-1)} \sum_{1 \leq i < j \leq n}  {\left( {\kappa}_{i} - {\kappa}_{j} \right)}^{2},
\]
we have the classical symmetric means inequality
\[
 \frac{H^2}{n^2} -    {\sigma}_{2} = \frac{1}{n^{2}(n-1)} \sum_{1 \leq i < j \leq n}  
{\left( {\kappa}_{i} - {\kappa}_{j} \right)}^{2} \geq 0.
\]
Applying the Hsiung--Minkowski formula \cite{Hsiung56} again, we obtain the integral identity
\[
   0 = \int_{\Sigma} \left( \frac{H}{n} + \frac{  {\sigma}_{2}  }{H} \IP{X}{\overrightarrow{\, H \,} } \right) d\Sigma
     = \int_{\Sigma} \left( \frac{H}{n} -  \frac{ n {\sigma}_{2}  }{H}  \right) d\Sigma
= \int_{\Sigma}  \frac{n}{H} \left(  \frac{H^2}{n^2} -    {\sigma}_{2}     \right) d\Sigma.
\]
Hence, $  \frac{H^2}{n^2} -    {\sigma}_{2} $ vanishes on $\Sigma$, which implies that 
${\kappa}_{1}=\cdots={\kappa}_{n}$ on $\Sigma$. Since ${\Sigma}^{n \geq 2}$ is a closed umbilic hypersurface in Euclidean space, it is a hypersphere. It follows from equation (\ref{rigidity proof 02}) that hypersphere is centered at the origin.
\end{proof}

\begin{lemma} \label{spherical expanders}
A minimal submanifold of the hypersphere ${\mathbb{S}}^{N\geq2} \subset {\mathbb{R}}^{N+1 \geq 3}$ is an expander for the inverse mean curvature flow. 
\end{lemma}

\begin{proof} 
Let $\Sigma^{n \geq 1}$ be a minimal submanifold of the hypersphere ${\mathbb{S}}^{N\geq2} \subset {\mathbb{R}}^{N+1 \geq 3}$, and let $X$ denote the position vector field.
On the one hand, since $X$ is already normal to the hypersphere ${\mathbb{S}}^{N} \subset {\mathbb{R}}^{N+1}$, we observe the equality
    $ {X}^{\perp} :=   {X}^{\perp \left( \Sigma \subset   {\mathbb{R}}^{N+1} \right)} = X$. 
    On the other hand, according to the minimality of  ${\Sigma}^{n}$ in ${\mathbb{S}}^{N}$, we obtain 
\begin{equation}  \label{mcv in sphere}
     {\triangle}_{g} X +  n X = 0,
\end{equation}
where $g$ denotes the induced metric on  ${\Sigma}^{n}$. Thus, we have 
\begin{equation}  \label{mcv in Euclidean space}
   \overrightarrow{\, H \,}  :={  \overrightarrow{\, H \,} }_{  \Sigma \subset   {\mathbb{R}}^{N+1}  }     (X)  =  {\triangle}_{g} X  =  - n X
\quad \text{and} \quad
 \vert     \overrightarrow{\, H \,}  \vert   = n \vert X \vert = n.
\end{equation}
Combining the four equalities on $\Sigma$ and taking $C=\frac{1}{n}>0$, we meet 
$ -  \frac{1}{{\vert \overrightarrow{\, H \,} \vert}^{2} \; } \overrightarrow{\, H \,}  = C  {X}^{\perp}$,
which indicates that ${\Sigma}^{n}$ is an expander for the inverse mean curvature flow. 
\end{proof}

\begin{theorem}
For any integer $g \geq 1$, there exists at least one two-dimensional compact embedded expander of genus $g$ in ${\mathbb{R}}^{4}$.
\end{theorem}
\begin{proof}
For any integer $g$, Lawson \cite{Lawson1970} showed that there exists a compact embedded minimal surface  ${\Sigma}$ of genus $g$  in ${\mathbb{S}}^{3}$. Lemma \ref{spherical expanders} shows that  ${\Sigma}$ becomes an expander to the inverse mean curvature flow in  ${\mathbb{R}}^{4}$.
\end{proof}


\section{Expanders with rotational symmetry}
\label{SSec4}

In this section, we investigate homothetic solitons in ${\mathbb{R}}^{n+1 \geq 3}$ with rotational symmetry about a line through the origin. Given a profile curve $\mathcal{C}$ parameterized by $\left( r(t), h(t) \right)$, $t \in I$ in the half-plane $\left\{  \; (r, h) \; \vert  \;   r>0, h \in \mathbb{R} \; \right\}$, we associate the induced rotational hypersurface in $ {\mathbb{R}}^{ n +1 }$ defined by
\[
   {\Sigma}^{n} =  \left\{ \;  X =  \bigg( r(t) \, {\mathbf{p}}, 
 h(t)   \bigg) 
 \in {\mathbb{R}}^{  n +1 } \, \, \bigg|    \; \left( r(t), h(t) \right) \in \mathcal{C},  \, {\mathbf{p}} \in   {\mathbb{S}}^{n-1} 
\subset {\mathbb{R}}^{n} \right\}.
\]
The rotational hypersurface $ {\Sigma}$ satisfies the homothetic soliton equation~(\ref{SOLhyper2}) if and only if the profile curve $\left( r(t), h(t) \right)$ satisfies the ODE
\begin{equation}
\label{ROTprofile}
-\left(  \frac{  \dot{r} \ddot{h}  -  \dot{h} \ddot{r}  }{ {\left(    {\dot{r} }^{2}   +   {\dot{h} }^{2}  \right)}^{\frac{3}{2}  }    }  +  \frac{    n-1    }  {  {\left(    {\dot{r} }^{2}   +   {\dot{h} }^{2}  \right)}^{\frac{1}{2}  }    }  \cdot \frac{  \dot{h}   }{  r  }  \right) \;
\frac{  - \dot{h} r + \dot{r}  h } {  {\left(    {\dot{r} }^{2}   +   {\dot{h} }^{2}  \right)}^{\frac{1}{2}  }     } = \frac{1}{C}
\end{equation}
for some constant $C>0$. We observe: 

\begin{enumerate}[i.]

\item
As long as the quantity $r \dot{h} - h\dot{r}$ is non-zero, we may write equation~(\ref{ROTprofile}) as
\[
\frac{  \dot{r} \ddot{h}  -  \dot{h} \ddot{r}  }{ {\dot{r} }^{2} + {\dot{h} }^{2} } = - \frac{ (n-1)  }{  r  }  \dot{h}  + \frac{{\dot{r} }^{2}   +   {\dot{h} }^{2}}{ C (r\dot{h} - h \dot{r}) } .
\]

\item
 The ODE (\ref{ROTprofile}) is invariant under the dilation $(r, h) \mapsto (\lambda r, \lambda h)$, 
unlike the profile curve equation for shrinkers or expanders for the mean curvature flow.  
\item Spheres are expanders. The half circle $\left( r(t), h(t) \right)=( R \cos t, R \sin t ), t \in \left( -\frac{\pi}{2}, \frac{\pi}{2} \right)$ having the origin as its center obeys the ODE (\ref{ROTprofile}). Indeed, we compute
 \[
  \frac{  \dot{r} \ddot{h}  -  \dot{h} \ddot{r}  }{ {\left(    {\dot{r} }^{2}   +   {\dot{h} }^{2}  \right)}^{\frac{3}{2}  }    }  
= \frac{1}{R}, \; \;
 \frac{    n-1    }  {  {\left(    {\dot{r} }^{2}   +   {\dot{h} }^{2}  \right)}^{\frac{1}{2}  }    }  \cdot \frac{  \dot{h}   }{  r  }    = \frac{n-1}{R}, \; \;
\frac{  - \dot{h} r + \dot{r}  h } {  {\left(    {\dot{r} }^{2}   +   {\dot{h} }^{2}  \right)}^{\frac{1}{2}  }     } = -R
\]
implies
\[
 - \left(  \frac{  \dot{r} \ddot{h}  -  \dot{h} \ddot{r}  }{ {\left(    {\dot{r} }^{2}   +   {\dot{h} }^{2}  \right)}^{\frac{3}{2}  }    }  +  \frac{    n-1    }  {  {\left(    {\dot{r} }^{2}   +   {\dot{h} }^{2}  \right)}^{\frac{1}{2}  }    }  \cdot \frac{  \dot{h}   }{  r  }  \right) \;
\frac{  - \dot{h} r + \dot{r}  h } {  {\left(    {\dot{r} }^{2}   +   {\dot{h} }^{2}  \right)}^{\frac{1}{2}  }     } = n.
\]

\item The lines $r(t)=\text{constant}$ are solutions to the ODE (\ref{ROTprofile}) when $C=1/(n-1)$. Cylinders become expanders.

\item We outline a way to deduce the ODE~(\ref{ROTprofile}) using the homothetic soliton equation $${\triangle}_{g}  {\vert X \vert}^{2}    =  2 \left( n - \frac{1}{C} \right).$$ We observe that $\Sigma$ is a homothetic soliton with rotational symmetry if and only if 
\begin{equation}  \label{div eqn}
2 \left(   n - \frac{1}{C}  \right) = {\triangle}_{g} \left(   r^{2}+h^{2}    \right) = \frac{1}{ r^{n-1}   {\left(    {\dot{r} }^{2}   +   {\dot{h} }^{2}  \right)}^{\frac{1}{2}  }   } \;  \frac{d}{dt} \left(  \; \frac{  r^{n-1}  }{  {\left(    {\dot{r} }^{2}   +   {\dot{h} }^{2}  \right)}^{\frac{1}{2}  }     }    \frac{d}{dt} \left( 
 r^{2}+h^{2}  \right) \;  \right),
\end{equation}
which is equivalent to (\ref{ROTprofile}).
\end{enumerate}

\subsection{Construction of expanding infinite bottles}
\label{construct_bottles}

Writing the profile curve $\mathcal{C}$ as a graph  $(r(h),h)$, we have the following second order non-linear differential equation:
\begin{equation}
\label{rot:eq1}
\frac{r''}{1+r'^2} = \frac{n-1}{r} - \frac{1+r'^2}{C(r - hr')}.
\end{equation}
When $C=\frac{1}{n-1}$, this equation becomes: 
\begin{equation}
\label{rot:eq1C}
\frac{r''}{1+r'^2} = (n-1) \left[ \frac{1}{r} - \frac{1+r'^2}{r - hr'} \right].
\end{equation}
Observe that $r(h) = constant$ is a solution to~(\ref{rot:eq1C}), which corresponds to a round hypercylinder expander. Moreover, if $r(h)$ is a solution to~(\ref{rot:eq1C}) with $r'(a) = 0$ for some $a\in \mathbb{R}$, then $r(h) \equiv r(a)$. Consequently, any nonconstant solution to~(\ref{rot:eq1C}) must be strictly monotone.

In this section, we construct new examples of entire solutions to (\ref{rot:eq1C}), which correspond to hypercylinder expanders that interpolate between two concentric round hypercylinders:

\begin{theorem}[\textbf{Construction of infinite bottles}]
\label{chimney_thm}
Let $r_0$, $h_0$, and $r_0'$ be constants satisfying: $r_{0}>0$, $h_{0}<0$, and $r_{0}' \in (0, -h_{0}/r_{0})$, and let $r(h)$ be the unique solution to (\ref{rot:eq1C}) satisfying the initial conditions: $r(h_{0})=r_{0}$ and $r'(h_{0})=r_{0}'$. Then $r(h)$ is an entire solution, and there are constants $0 < r_{bot} < r_{top} < \infty$ so that $r(h)$ interpolates between $r_{bot}$ and $r_{top}$. More precisely, $r(h)$ is strictly increasing, $\lim_{h \to -\infty}r(h) = r_{bot}$, $\lim_{h \to \infty}r(h) = r_{top}$, and there exists a point $h_1 \in (h_0, 0)$ so that $r''(h_1)=0$ and $r''(h)$ has the same sign as $(h_1-h)$ when $h \neq h_1$.
\end{theorem}

\begin{proof}[Proof of Theorem~\ref{chimney_thm}]

We separate the proof into two parts. First, we show that the solution is entire and increasing, and there is a unique point where the concavity changes sign. Second, we establish estimates that bound the solution between two positive constants. We note that the rotation of the profile curve about the $h$-axis has the appearance of an
infinite bottle, which interpolates between two concentric cylinders.

\subsubsection*{{\bf Part 1:} Existence of expanding infinite bottles}

Let $r_0$, $h_0$, and $r_0'$ be constants satisfying: $r_{0}>0$, $h_{0}<0$, and $r_{0}' \in (0, -h_{0}/r_{0})$, and let $r(h)$ be the unique solution to (\ref{rot:eq1C}) satisfying the initial conditions: $r(h_{0})=r_{0}$ and $r'(h_{0})=r_{0}'$.

Notice that the condition $r'(h_0) = r_0' >0$ shows that $r$ is a nonconstant solution and guarantees that $r'(h)>0$. Also, observe that the assumption $r_{0}' \in (0, -h_0/r_0)$ coupled with the defining initial conditions for $r(h)$ shows that $h+r'r$ is negative at $h=h_0$. In fact, by assumption, the terms $r'$, $-h-r'r$, $r$, and $r-hr'$ are all positive at $h=h_0$. So, writing equation (\ref{rot:eq1C}) as
\begin{equation}
\label{rot:eq1C re}
  r'' = (n-1) (1+r'^2 )  \frac{r'(-h-r'r) }{  r(r-hr') },
\end{equation}
we see that $r''(h_{0})>0$.

In the following lemma, we show that the concavity of $r(h)$ changes sign exactly once when $r(h)$ is a maximally extended solution.
\begin{lemma}[\textbf{Existence of a unique inflection point}]
\label{inflection existence}
Let $r_0$, $h_0$, and $r_0'$ be constants satisfying: $r_{0}>0$, $h_{0}<0$, and $r_{0}' \in (0, -h_{0}/r_{0})$, and let $r(h)$ be the solution to (\ref{rot:eq1C}) satisfying the initial conditions: $r(h_{0})=r_{0}$ and $r'(h_{0})=r_{0}'$. If $r(h)$ is a maximally extended solution, then there exists a point $h_1 \in (h_0,0)$ so that $r''(h_1)=0$. Furthermore, $r''(h)$ has the same sign as $(h_1 - h)$ when $h \neq h_1$.
\end{lemma}
 
\begin{proof}
Let $r(h)$ be a maximally extended solution to (\ref{rot:eq1C}) satisfying the above assumptions. Then  there are constants $h_{min}$ and $h_{max}$ satisfying $-\infty \leq h_{min} < h_0 < h_{max} \leq \infty$ so that $r(h)$ is defined for all $h \in (h_{min}, h_{max})$. It follows from the preceding paragraph that $r''(h_{0})>0$.

\textbf{Step A.} We claim that there exists a point $h_1 \in (h_0, 0)$ so that $r''(h_1) = 0$. We first treat the case where $h_{max} \leq 0$. In this case, proving the claim is equivalent to showing there is a point $h_1 \in (h_0, h_{max})$ so that $r''(h_1) = 0$. Suppose to the contrary that 
\[
r''(h) > 0 \quad \text{for all} \;\; h \in (h_{0}, h_{max}). 
\]
As $h_{max} \leq 0$ and both $r$ and $r'$ are positive, we have $(r-hr')>0$ for $h \in (h_{0}, h_{max})$. In fact, since $\frac{d}{dh}(r-hr') = -hr''>0$, we see that $(r-hr')>r_0-h_0r_0'$. Using equation (\ref{rot:eq1C re}) and the positivity of the functions $r$, $r'$, $(r-hr')$, and $r''$, we arrive at the inequality $(-h-rr') > 0$, which leads to the estimate
\[
0 < r'(h) < -\frac{h}{r} < -\frac{h_0}{r_0} \quad \text{for all} \;\; h \in (h_{0}, h_{max}). 
\]
Now, returning to equation (\ref{rot:eq1C re}), we have the estimate
\[
0 \leq r''(h) = (n-1) (1+r'^2 )  \frac{r'(-h-r'r) }{  r(r-hr') } \leq (n-1)\left( 1+ \left( \tfrac{h_0}{r_0}\right)^2 \right) \frac{(-\frac{h_0}{r_0}) (-h_0)}{r_0(r_0-h_0r'_0) },
\]
for $h \in (h_{0}, h_{max})$. These estimates contradict the finiteness of the maximal endpoint $h_{max}$, and we conclude that the claim is true in the case where $h_{max} \leq 0$.

It still remains to prove the claim in the case where $h_{max} >0$. However, in this case the solution $r(h)$ is defined when $h=0$ and equation (\ref{rot:eq1C}) implies  
 \[
 r''(0)= - (n-1) \frac{{r'(0)}^2}{r(0)} <0. 
 \]
It follows that there exists a point $h_1 \in (h_0, 0)$ so that $r''(h_1)=0$.

\textbf{Step B.} We claim that $r''(h)$ has the same sign as $h_1-h$. Taking a derivative of equation~(\ref{rot:eq1}), we have
\[
\frac{r'''}{1+r'^2} = \frac{2r'(r'')^2}{(1+r'^2)^2} -\frac{n-1}{r^2}r' - \frac{2r'r''}{C(r - hr')} - \frac{1+r'^2}{C(r - hr')^2}hr''.
\]
At the point $h_1$, we obtain    
\[
\frac{r'''(h_1)}{1+r'(h_1)^2} = - (n-1) \frac{r'(h_1)}{{r(h_1)}^2}<0,
\]
which shows that $r''(h)$ has the same sign as $h_1-h$ in a neighborhood of $h_1$. In fact, at any point $\bar{h}$ where $r''(\bar{h})=0$, we have $r'''(\bar{h})<0$. This property tells us that the sign of $r''$ can only change from positive to negative, and consequently $r''$ vanishes at most once. Thus, $r''(h)$ has the same sign as $h_1-h$ for all $h \in (h_{min}, h_{max})$.
\end{proof}
 
Next, we prove that the profile curves corresponding to the infinite bottles come from entire graphs.

\begin{lemma}[\textbf{Existence of entire solutions}]
Let $r_0$, $h_0$, and $r_0'$ be constants satisfying: $r_{0}>0$, $h_{0}<0$, and $r_{0}' \in (0, -h_{0}/r_{0})$, and let $r:(h_{min},h_{max}) \to \mathbb{R}_+$ be the maximally defined solution to (\ref{rot:eq1C}) satisfying the initial conditions: $r(h_{0})=r_{0}$ and $r'(h_{0})=r_{0}'$. Then $h_{max}=\infty$ and $h_{min} = -\infty$.
\end{lemma}

\begin{proof}
Let $r(h)$ be a maximally extended solution to (\ref{rot:eq1C}) satisfying the above assumptions. In the previous lemma we proved the existence of a point $h_1 \in (h_0,0)$ so that $r''(h_1)=0$ and $r''(h)$ has the same sign as $(h_1-h)$ when $h \neq h_1$.

\textbf{Step A.}  We claim that $h_{max}=\infty$. First, we show that $h_{max}>0$. To see this, notice that $0 \leq r'(h) \leq r'(h_1)$, $r(h) \geq r_0$, and $r-hr' \geq r_0$ whenever $h_1 \leq h \leq 0$. It follows from equation (\ref{rot:eq1C}) that the solution $r(h)$ can be extended past $h \leq 0$. Thus, $h_{max}>0$. Next, we show that $h_{max} = \infty$. Since $h_1<0$, we have $\frac{d}{dh}(r-hr') = -hr'' \geq 0$ when $h \geq 0$ so that $(r-hr') \geq r(0)$ when $h \geq 0$. We also have $0 \leq r'(h) \leq r'(h_1)$ and $r(h) \geq r_0$ when $h \geq 0$. As before, it follows from equation (\ref{rot:eq1C}) that the solution $r(h)$ can be extended past any finite point.

\textbf{Step B.} We claim that $h_{min} = -\infty$. Suppose to the contrary that $h_{min} > -\infty$. Then at least one of the functions $r'$, $\frac{1}{r}$, $\frac{1}{r-hr'}$ must  blow-up at the finite point $h=h_{min}$. Since $r''>0$ on $(h_{min}, h_1)$, the positive function $r'$ is increasing, and we have  $  r'(h)  \leq  r'(h_0) =r_0'$ for all $h \in (h_{min}, h_0)$. So, the function $r'$ does not blow-up at $h_{min}$. If the function $1/r$ is bounded above on $(h_{min}, h_0)$, then the inequality $0<r(h)<r(h)-h\,r'(h)$ (when $h \leq 0$) guarantees that $1/(r-hr')$ is also bounded above on $(h_{min}, h_0)$, in which case, the solution can be extended prior to $h_{min}$. Therefore, the function  $\frac{1}{r}$ must blow-up at $h=h_{min}$. In other words, we have 
 \[
 \lim_{h \to {h_{min}}^{+}} r(h)=0.
 \] 
Observing this and using $0<r'(h)<r_0'$ on $(h_{min}, h_0)$, we can find a sufficiently small $\delta>0$ so that $r'(h)r(h) \geq \frac{-h_{0}}{2}$ for all 
$h \in (h_{min}, h_{min}+\delta]$. Also, the inequality $\frac{d}{dh}(r-hr')=-hr''>0$ guarantees that $0<r(h)-h r'(h) \leq \epsilon_1 :=r( h_{min}+\delta)-( h_{min}+\delta) r'( h_{min}+\delta)$. 
It follows from these estimates and equation (\ref{rot:eq1C}) that 
 \[
  \frac{d}{dh} \left( \arctan r' \right)= \frac{r''}{1+r'^2} =(n-1) \frac{-(h+r'r)}{r-hr'} \cdot \frac{r'}{r} \geq \epsilon_{2}   \frac{d}{dh} \left( \ln r \right),
\]
where $\epsilon_2 =\frac{(n-1) \frac{-h_{0}}{2}}{\epsilon_1}>0$ is a constant. Hence, the function $F(h):=\arctan \left( \frac{dr}{dh} \right) -\epsilon_2 \ln r(h)$ is increasing on $(h_{min}, h_{min}+\delta]$. Thus, we have the estimate 
\[
 \epsilon_2 \ln r(h)  \geq -  F\left( h_{min}+\delta \right) + \arctan r' > -  F\left( h_{min}+\delta \right). 
 \]
Taking the limit as $h \to h_{min}^+$ and using $\lim_{h \to {h_{min}}^{+}} r(h)=0$ leads to a contradiction. We conclude that $h_{min}=-\infty$. 
\end{proof}

So far, we have proved the existence of an entire bottle solution $r(h)$ to (\ref{rot:eq1C}). In the next part of the proof we will establish estimates that squeeze the ends of the infinite bottles between two cylinders.
 
\subsubsection*{{\bf Part 2:} Squeezing infinite bottles by two hypercylinders}

To establish upper and lower bounds for the solution $r(h)$, we study the profile curve $\mathcal{C}$ by writing it as a graph over the axis of rotation: $(r,h(r))$. Then, we have the following second order non-linear differential equation
\begin{equation}
\label{rot:h:eq1}
\frac{h''}{1+h'^2} = -\frac{(n-1)}{r} h' + \frac{1+h'^2}{C(rh' -h)},
\end{equation}
or equivalently,
\begin{eqnarray}
\label{rot:h:eq1.1}
\frac{h''}{1+h'^2}= \frac{(n-1)h h' + \frac{1}{C}r}{r(rh'-h)} + \left( \frac{1}{C} - (n-1) \right) \frac{h'^2}{(rh' -h)}         \nonumber
\end{eqnarray}
Throughout this section, we take $C=\frac{1}{n-1}$ so that equation~(\ref{rot:h:eq1}) takes the form
\begin{equation}
\label{rot:h:eq:c}
\frac{h''}{1+h'^2} = -(n-1) \left[ \frac{h'}{r} - \frac{1+h'^2}{rh' - h } \right].
\end{equation}

\begin{lemma}[\textbf{Existence of the outside cylinder barrier}]
\label{outside}
Let $h(r)$ be a maximally extended solution to~(\ref{rot:h:eq:c}) defined on $(r_{bot},r_{top})$. Assume there is a point $r_1 \in (r_{bot},r_{top})$ so that $h'(r) > 0$ and $h''(r)>0$ for all $r \in (r_1, r_{top})$. Also, assume that $r_1 h'(r_1) - h(r_1) > 0$. Then, we have  
\[
r_{top} < \infty, \quad \lim_{r \to {r_{top}}^{-} } h'(r) = \infty, \quad \text{and} \lim_{r \to {r_{top}}^{-}} h(r) = \infty.
\]
\end{lemma}
\begin{proof} 
We introduce the angle functions $\theta, \phi : (r_1, r_{top}) \rightarrow \left(0, \frac{\pi}{2}\right]$, defined by 
\[
\theta(r)=\arctan \left( \frac{dh}{dr} \right) \quad \text{and} \quad \phi(r)=\arctan \left( \frac{h}{r} \right),
\]
to rewrite the profile curve equation (\ref{rot:h:eq:c}) as 
\begin{equation} \label{angles eqn}
\frac{d\theta}{dr} =  \frac{n-1}{r \cdot \tan\left(\theta -\phi\right)}.
\end{equation}
Combining this and $0<\tan\left(\theta -\phi\right) \leq \tan \theta$, we have $\frac{d\theta}{dr} \geq   \frac{n-1}{r \cdot \tan \theta}$, which implies 
\[
\frac{d}{dr} \left( \frac{\tan \theta}{r^{n-1}} \right) \geq \frac{n-1}{r^{n} \tan \theta} \geq 0.
\]
This tells us that the continuous function $\frac{\tan \theta}{r^{n-1}}$ is increasing for $r > r_{1}$. According to the estimate  
\[
\frac{d}{dr} \left(  h -  \frac{  \tan \theta_1 }{n {r_{1}}^{n-1} } r^n  \right) = \tan \theta -  \frac{ \tan \theta_{1} }{ {r_{1}}^{n-1} }  r^{n-1} = \left( \frac{ \tan \theta }{ {r}^{n-1} } - \frac{ \tan \theta_{1} }{ {r_{1}}^{n-1} } \right) r^{n-1} \geq 0,
\]
we see that the function $(h -  \frac{  \tan \theta_{1} }{n {r_{1}}^{n-1} } r^n)$ is increasing. In particular, we have the height estimate
\[
h \geq h_{1} +   \frac{  \tan \theta_{1} }{n {r_{1}}^{n-1} } \left(  r^n - {r_{1}}^{n} \right).
\]
Observe that 
$ \frac{1}{\tan\left(\theta -\phi\right)} = \frac{1 + \tan \theta \tan \phi}{\tan \theta - \tan \phi} \geq \tan \phi$. Combining this  with equation (\ref{angles eqn}) we have 
\[
    \frac{1}{n-1} \frac{d\theta}{dr}  \geq \frac{\tan \phi }{r} = \frac{h}{r^2} \geq \frac{1}{r^2} \left( h_{1} + 
     \frac{  \tan \theta_{1} }{n {r_{1}}^{n-1} } \left(  r^n - {r_{1}}^{n} \right) \right),
\]
which implies
\[ 
  \frac{d}{dr} \left[ \frac{\theta}{n-1}   +    \left(       h_{1} -  \frac{\tan \theta_{1}}{n} r_1   \right)  \frac{1}{r}     - \frac{\tan \theta_{1}}{ n (n-1) \, {r_{1}}^{n-1}  }  r^{n-1}   \right] \geq 0.
\] 
Therefore, the function $F(r)=\frac{\theta}{n-1}   +    \left(       h_{1} -  \frac{\tan \theta_{1}}{n} r_1   \right)  \frac{1}{r}     - \frac{\tan \theta_{1}}{ n (n-1) \, {r_{1}}^{n-1}  }  r^{n-1}$ is increasing, and for all $r \in (r_1, r_{top})$, we have
\[
 \frac{\theta}{n-1}   \geq   F \left(r_{1} \right) -  \left(       h_{1} -  \frac{\tan \theta_{1}}{n} r_1   \right)  \frac{1}{r}     + \frac{\tan \theta_{1}}{ n (n-1) \, {r_{1}}^{n-1}  }  r^{n-1}.
\] 
Since the left hand side is bounded above, and the right hand side becomes arbitrarily large as $r$ goes to $\infty$, we conclude that $r_{top}<\infty$. It then follows that the increasing, concave up function $h(r)$ satisfies $\lim_{r \to {r_{top}}^-} h'(r) = \infty$. If $h(r)$ has a finite limit as $r$ approaches $r_{top}$, then by the uniqueness of cylnder $r(h) \equiv r_{top}$, we get a contradiction. Therefore, we also have $\lim_{r \to {r_{top}}^{-} } h(r) = \infty$.
\end{proof}

Next, we prove the following lemma, which shows that a solution with $h<0$, $h'>0$ and $h''<0$ cannot approach the axis of rotation.

\begin{lemma}[\textbf{Existence of the inside cylinder barrier}]
\label{inside}
Let $h(r)$ be a maximally extended solution to~(\ref{rot:h:eq:c}) defined on $(r_{bot},r_{top})$. Assume there is a point $r_0 \in (r_{bot}, r_{top})$ so that $h(r)<0$, $h'(r) > 0$ and $h''(r) < 0$ for all $r \in (r_{bot}, r_{0}]$. Then, we have  
\[
r_{bot} >0, \quad \lim_{r \to {r_{bot}}^{+} } h'(r) = \infty, \quad \text{and} \lim_{r \to {r_{bot}}^{+}} h(r) =- \infty.
\]
\end{lemma}

\begin{proof} 
We first observe that  $h -rh'<0$ and $h h'<0$.  We introduce three well-defined functions $\theta : (r_{bot}, r_{0}] \rightarrow \left(0, \frac{\pi}{2}\right]$ and ${\Psi}_{1}, {\Psi}_{2}: (r_{bot}, r_{0}]  \rightarrow \mathbb{R}$ as folllows: 
\[
\theta(r)=\arctan \left( \frac{dh}{dr} \right), \quad {\Psi}_{1}(r) = \frac{ -h h' }{ rh'-h }  , \quad \text{and} \quad {\Psi}_{2}(r)= \frac{r+h h'}{h h'},
\]
and we rewrite the profile curve equation (\ref{rot:h:eq:c}) as 
\begin{equation} \label{angles eqn2}
\frac{d\theta}{dr} =  -\frac{n-1}{r} {\Psi}_{1} {\Psi}_{2}.
\end{equation}
Using the estimate $\frac{d{\Psi}_{1}}{dr} = \frac{  -r (h')^{3} + h (  (h')^2+ h h''  ) }{(h - rh')^2} \leq 0$, 
we see that ${\Psi}_{1}$ is decreasing on $(r_{bot}, r_{0}]$, and setting $\epsilon_{1}= {\Psi}_{1}(r_{0})$, we have 
\begin{equation} \label{F estimate}
 {\Psi}_{1}(r) \geq \epsilon_{1}>0.
\end{equation}

Observing $(h h')' =  {h'}^{2}+h''h >0$ and defining a positive constant $\epsilon_{2}=-h(r_0)h'(r_0)$, we have 
the estimate $hh' \leq -\epsilon_{2}$ for all $r \in (r_{bot}, r_{0}]$. It follows that 
\begin{equation} \label{G estimate}
{\Psi}_{2}(r)= 1+ \frac{r}{h h'} \geq 1 - \frac{r}{{\epsilon}_{2}}.
\end{equation}
Combining  (\ref{angles eqn2}), (\ref{F estimate}), and (\ref{G estimate}), we have
\[
\frac{d}{dr} \left(  \frac{\theta}{(n-1) \epsilon_1} + \ln r - \frac{r}{\epsilon_2} \right) \leq 0.
\]
Therefore, the function $\Psi(r)=\frac{\theta}{(n-1) \epsilon_1} + \ln r - \frac{r}{\epsilon_2}$ is decreasing, and for all $r \in (r_{bot}, r_{0}]$, we have
\[
   \frac{\theta}{(n-1) \epsilon_1}  \geq -\ln r + \frac{r}{\epsilon_2} + \Psi(r_{0}).
\]
Since the left hand side is bounded above, and the right hand side becomes arbitrarily large as $r$ goes to $0$, we conclude that $r_{bot}>0$. It then follows that the increasing, concave down function $h(r)$ satisfies $\lim_{r \to {r_{bot}}^+} h'(r) = \infty$. If $h(r)$ has a finite limit as $r$ approaches $r_{bot}$, then by comparison with the cylnder $r(h) \equiv r_{bot}$, we get a contradiction. Therefore, we also have $\lim_{r \to {r_{bot}}^{+} } h(r) = -\infty$.
\end{proof}

In conclusion, for constants $r_0$, $h_0$, and $r_0'$ satisfying: $r_{0}>0$, $h_{0}<0$, and $r_{0}' \in (0, -h_{0}/r_{0})$, there exists an entire solution $r(h)$ to (\ref{rot:eq1C}) satisfying the initial conditions: $r(h_{0})=r_{0}$ and $r'(h_{0})=r_{0}'$. Moreover, there are constants $0 < r_{bot} < r_{top} < \infty$ so that $r(h)$ interpolates between $r_{bot}$ and $r_{top}$ in the sense that $r(h)$ is strictly increasing, $\lim_{h \to -\infty}r(h) = r_{bot}$, $\lim_{h \to \infty}r(h) = r_{top}$, and there exists a point $h_1 \in (h_0, 0)$ so that $r''(h_1)=0$ and $r''(h)$ has the same sign as $(h_1-h)$ when $h \neq h_1$.

\end{proof}

\subsection{Other examples of complete solitons}

In~\cite{HI1997}, Huisken and Ilmanen used a phase-plane analysis to exhibit complete, rotationally symmetric expanders for the inverse mean curvature flow, which are topological hyperplanes. For each $C > 1/n$, they showed there exists a half-entire solution to (\ref{rot:eq1}), which intersects the $h$-axis perpendicularly, and they provided numeric descriptions of these profile curves. For $C > 1/n$ and $C \neq 1/(n-1)$, they also indicated the existence of entire solutions to (\ref{rot:eq1}), which are symmetric about the $r$-axis and correspond to topological hypercylinders. (We note that the rotational expander constructed in Theorem~\ref{chimney_thm} is non-symmetric in the sense that its profile curve is not symmetric about the $r$-axis.) In this section, we explain how the techniques from Section~\ref{construct_bottles} can be used to recover the examples and numeric pictures preseneted in~\cite{HI1997}.

\subsubsection*{Hyperplane expanders}

We begin by considering the initial value problem where we shoot perpendicularly to the axis of rotation. For $C>0$, let $h(r)$ be a solution to (\ref{rot:h:eq1}) with $h(0)=h_0<0$ and $h'(0)=0$. This singular shooting problem is well-defined (see~\cite{BG1976} and~\cite{D2015}), and the solution satisfies $h''(0) = -1 / (n C h_0) > 0$. Differentiating (\ref{rot:h:eq1}) and analyzing the equation for $h'''(r)$ shows that, under the above conditions, we have $h''(r)>0$ and $h'(r)>0$, for $r>0$, as long as the solution is defined. The global behavior of the solution ultimately depends on the value of $C$.

When $h(r)$ is a solution to the above shooting problem, the graph $(r,h(r))$ is part of a profile curve $\mathcal{C}$, which corresponds to a rotational expander for the inverse mean curvature flow. Applying the techniques from the proof of Theorem~\ref{chimney_thm} to the profile curve $\mathcal{C}$ leads to a description of the global behavior of this expander, which ultimately depends on the value of $C > 1/n$. In terms of the profile curve $\mathcal{C}$ written as a graph over the $h$-axis, we have the following result.

\begin{theorem}
\label{shooting_thm}
For $C > 1/n$ and $h_0<0$, there exists a half-entire solution $r(h)$ to (\ref{rot:eq1}) that is defined for $h>h_0$, and such that the curve $(h,r(h))$ intersects the $h$-axis perpendicularly when $h=h_0$. The solution $r(h)$ has three types of behavior, depending on the value of $C$:
\begin{enumerate}
\item If $C=1/(n-1)$, then $r'>0$, $r''<0$, and there exists $0<r_{top} < \infty$ so that $\lim_{h \to \infty} r(h) = r_{top}$. 

\item If $C > 1/(n-1)$, then $r'>0$, $r''<0$, and $\lim_{h \to \infty} r(h) = \infty$. 

\item If $1/n < C < 1 / (n-1)$, then there exists a point $h_1$ so that $r''(h)$ has the same sign as $(h-h_1)$, and $\lim_{h \to \infty} r(h) = 0$.
\end{enumerate}
\end{theorem}
\begin{proof}
When $C=1/(n-1)$, the convexity of $h(r)$, along with the analysis from Lemma~\ref{outside} shows that there is a point $r_{top} < \infty$ so that $\lim_{r \to {r_{top}}^-} h'(r) =\infty$ and $\lim_{r \to {r_{top}}^-} h(r) =\infty$. Written as a graph over the $h$-axis, this shows that there is a solution $r(h)$ to (\ref{rot:eq1}), defined for $h>h_0$, which intersects the $h$-axis perpendicularly at $h_0$ and satisfies $r'>0$, $r''<0$, and $\lim_{h \to \infty} r(h) = r_{top}$. 

Next, when $C > 1 / (n-1)$, we claim that the solution $h(r)$ must exist for all $r>0$. To see this, suppose to the contrary that $h'$ increases to $\infty$ at a point $r_{top} < \infty$. Then, since $C>1/(n-1)$, equation (\ref{rot:h:eq1}) forces $h \geq \epsilon r h'$ when $r$ is close to $r_{top}$, for some $\epsilon>0$. However, integrating this inequality shows that $h'$ does not blow-up at a finite point; hence the solution exist for all $r>0$. Therefore, the solution $h(r)$ exists for all $r>0$, and using $h''>0$ and $h'>0$, we have $\lim_{r \to \infty} h(r) = \infty$. Written as a graph over the $h$-axis, this shows that there is a solution $r(h)$ to (\ref{rot:eq1}), defined for $h>h_0$, which intersects the $h$-axis perpendicularly at $h_0$ and satisfies $r'>0$, $r''<0$, and $\lim_{h \to \infty} r(h) = \infty$.

Finally, when $1/n < C < 1 / (n-1)$, the term $\left( \frac{1}{C} - (n-1) \right)$ is positive and the analysis in Lemma~\ref{outside} can be used to show that $h(r)$ does not exist for all $r>0$. Moreover, using the positivity of $\left( \frac{1}{C} - (n-1) \right)$ and integrating equation (\ref{rot:h:eq1}), we arrive at an inequality that provides an upper bound for $h$. In terms of the profile curve written as a graph over the $h$-axis, this says that the solution $r(h)$ achieves a global maximum at a finite point. 
Reading equation (\ref{ROTprofile}) in polar coordinates, we can show that $r(h)$ is defined for $h>h_0$. This forces the concavity of $r(h)$ to change sign at a finite point, and as in the proof of Lemma~\ref{inflection existence}, it follows that there is a point $h_1$ so that $r''(h)$ has the same sign as $(h-h_1)$. Then, an argument similar to the one in the previous paragraph shows that $r(h)$ is not bounded below by a positive constant, and we conclude that $\lim_{h \to \infty} r(h) = 0$.
\end{proof}

We remark that when $1/n < C < 1 / (n-1)$, the analogue of Lemma~\ref{outside} holds, but as we saw in the proof of the previous theorem, the analogue of Lemma~\ref{inside} is not true. Similarly, if $C>1/(n-1)$, then the analogue of Lemma~\ref{inside} holds, but the analogue of Lemma~\ref{outside} does not.

\subsubsection*{Hypercylinder expanders}

We finish this section with the following result on the construction of rotational expanders that are topological hypercylinders.

\begin{theorem}
\label{symmetric_thm}
For $C > 1/n$ and $r_0>0$, there is a unique solution $r(h)$ to (\ref{rot:eq1}) that is symmetric about the $r$-axis and satisfies the initial condition: $r(0)=r_0$, $r'(0)=0$. The solution $r(h)$ has three types of behavior, depending on the value of $C$:
\begin{enumerate}
\item If $C=1/(n-1)$, then $r(h) \equiv r_0$ (which corresponds to the round hypercylinder). 

\item If $C > 1/(n-1)$, then $r(h)$ has a global minimum at $h=0$, and there exists a point $h_1>0$ so that $r''(h)$ has the same sign as $(h_1 - |h|)$. Also, $\lim_{h \to \infty} r(h) = \infty$. 

\item If $1/n < C < 1 / (n-1)$, then $r(h)$ has a global maximum at $h=0$, and there exists a point $h_1>0$ so that $r''(h)$ has the same sign as $(|h| - h_1)$. Also, $\lim_{h \to \infty} r(h) = 0$. 
\end{enumerate}
\end{theorem}
\begin{proof}
It follows from equation (\ref{rot:eq1}) that the condition $r'(0)=0$ forces the solution to be constant when $C=1/(n-1)$, to have a global minimum at $h=0$ when $C > 1/(n-1)$, and to have a global maximum at $h=0$ when $1/n < C < 1 / (n-1)$. To see that there is a finite point $h_1>0$ where the concavity of $r(h)$ changes sign when $C > 1/ (n-1)$, we first observe that $r(h)$ is increasing when $h>0$, and consequently, it is defined for all $h>0$. An analysis of equation (\ref{rot:h:eq1}) shows that a positive solution $h(r)$ cannot satisfy $h''(r)<0$ and $h'(r)>0$ for all $r>0$ when $C > 1/(n-1)$; hence, there is a finite point $h_1>0$ where the concavity of $r(h)$ changes sign. When $1/n < C < 1 / (n-1)$, the analysis in the proof of Theorem~\ref{shooting_thm} can be used to show that the concavity of $r(h)$ changes sign at a finite point $h_1>0$. The proofs of the remaining properties are similar to the proofs given in Theorem~\ref{chimney_thm} and Theorem~\ref{shooting_thm}.
\end{proof}


\begin{thebibliography}{00}

\bibitem{AM2010}
T. M. Apostol, M. A. Mnatsakanian,   
\textit{Tanvolutes: generalized involutes},
 Amer. Math. Monthly, \textbf{117} (8), 701--713.
  
\bibitem{Andrews2003}
 B. Andrews, \textit{Classification of limiting shapes for isotropic curve flows}, 
  J. Amer. Math. Soc. \textbf{16} (2003), no. 2, 443-459.
   
\bibitem{AN2007}
K. Akutagawa, A. Neves, \textit{3-­manifolds with Yamabe invariant greater than
that of $\mathbb{R}\mathbb{P}^3$}, J. Differential Geom. \textbf{75} (2007),
no. 3, 359--386.

\bibitem{BG1976}
M. S. Baouendi, C. Goulaouic, 
\textit{Singular nonlinear Cauchy problems}, 
J. Differential Equations \textbf{22} (1976), no. 2, 268--291.
       
\bibitem{B2001}
H. Bray, \textit{Proof of the Riemannian Penrose inequality using the
positive mass theorem}, J. Differential Geom \textbf{59} (2001), 177--267.

\bibitem{B2002}
H. Bray, \textit{Black holes, geometric flows, and the Penrose inequality in general relativity},
Notices of the AMS \textbf{49} (2002), no. 11, 1372--1381.

\bibitem{BHMS2007}
H. Bray, S. Hayward, M. Mars, W. Simon,
\textit{Generalized inverse mean curvature flow in spacetimes},
Comm. Math. Phys. \textbf{272} (2007), 119--138.
 
\bibitem{BN2004}
  H. Bray, A. Neves, 
  \textit{Classification of prime 3-Manifolds with Yamabe invariant greater
than $\mathbb{R}\mathbb{P}^3$}, Ann. of Math. \textbf{159} (2004), 407--424.

\bibitem{D2015}
G. Drugan,
\textit{An immersed $S^2$ self-shrinker},
Trans. Amer. Math. Soc. \textbf{367} (2015), no. 5, 3139--3159.

\bibitem{G1990}
C. Gerhardt,
\textit{Flow of nonconvex hypersurfaces into spheres},
J. Differential Geom. \textbf{32} (1990), no. 1, 299--314.
 
\bibitem{G1973}
R. Geroch,
\textit{Energy extraction},
Ann. New York Acad. Sci. \textbf{224} (1973), 108--117.

 
 \bibitem{HI1997}
  G. Huisken, T. Ilmanen, \textit{A note on inverse mean curvature flow}, 
Proc. of the Workshop on Nonlinear Partial Differential Equations. Saitama Univ. (1997). 
 
\bibitem{HI1997_2}
G. Huisken, T. Ilmanen,
\textit{The Riemannian Penrose inequality},
Int. Math. Res. Not. \textbf{20} (1997), 1045--1058.
     
 \bibitem{HI2001}
 G. Huisken, T. Ilmanen, \textit{The inverse mean curvature flow and the Riemannian Penrose inequality},
  J. Differential Geom. \textbf{59} (2001), 353--437.
  
 \bibitem{HI2008}
 G. Huisken, T. Ilmanen, \textit{Higher regularity of the inverse mean curvature flow},
 J. Differential Geom. \textbf{80} (2008), 433--451.

 \bibitem{Hsiung56}
 C. C. Hsiung, \textit{Some integral formulas for closed hypersurfaces in Riemannian space},
 Pac. J. Math. \textbf{6} (no. 2) (1956), 291--299.
 
 
  \bibitem{KM2014}
  K.-K. Kwong, P. Miao, 
  \textit{A new monotone quantity along the inverse mean curvature flow in ${\mathbb{R}}^{n}$}, 
  Pacific J. Math. \textbf{267} (2014), no. 2, 417--422.
 
\bibitem{Lawson1970} 
H. B. Lawson, \textit{Complete minimal surfaces in ${\mathbb{S}}^{3}$},  Ann. of Math. (2) \textbf{92} (1970) 335--374.

    \bibitem{LWW2011}
 Y.-I. Lee, A.-N. Wang, S. W. Wei, 
 \textit{On a generalized 1­-harmonic equation and the inverse mean curvature flow},
 J. Geom. Phys. 61 (2011), no. 2, 453--461.  
   
   \bibitem{M2007}
 R. Moser, \textit{The inverse mean curvature flow and ­ $p$-harmonic functions},
J. Eur. Math. Soc. \textbf{9} (2007), no. 1, 77--83.

\bibitem{U1990}
J. Urbas, \textit{On the expansion of starshaped hypersurfaces by symmetric
functions of their principal curvatures}, Math. Z. \textbf{205} (1990),
355-372.

\bibitem{S2000}
K. Smoczyk, \textit{Remarks on the inverse mean curvature flow}, Asian J. Math., \textbf{4} (2000), no. 2, 331--336.


\end{thebibliography}
\end{document}